\documentclass[reqno]{amsart}

\usepackage{amssymb}

\newtheorem{theorem}{Theorem}[section]

\newtheorem{lemma}{Lemma}[section]

\theoremstyle{definition}
\newtheorem{definition}{Definition}[section]
\newtheorem{remark}{Remark}[section]

\numberwithin{equation}{section}

\DeclareMathOperator{\supp}{supp}

\DeclareMathOperator{\essinf}{ess\, inf}

\newcommand\R{\mathbb R}

\newcommand\C{\mathbb C}
\newcommand\Z{\mathbb Z}

\newcommand\btau{\boldsymbol{\tau}}

\newcommand\bdeta{\boldsymbol{\eta}}

\newcommand{\cA}{\mathcal{A}}

\newcommand{\bom}{\boldsymbol{\omega}}

\newcommand\eps{\varepsilon}

\renewcommand\P{\mathbb P}
\newcommand\E{\mathbb E}

\newcommand{\abs}[1]{\left\lvert #1 \right\rvert}

\newcommand{\set}[1]{\left\{ #1 \right\}}
\newcommand{\pa}[1]{\left( #1 \right)}

\newcommand{\eq}[1]{\eqref{#1}}

\newcommand\beq{\begin{equation}}
\newcommand\eeq{\end{equation}}

\newcounter{numcount}
\newcommand{\labelnummer}{\mbox{\normalfont (\roman{numcount})}}%

\makeatletter

\newenvironment{nummer}%
  {\let\curlabelspeicher\@currentlabel%
    \begin{list}{\labelnummer}%
      {\usecounter{numcount}\leftmargin0pt%
        \topsep0.5ex\partopsep2ex\parsep0pt\itemsep0ex\@plus1\p@%
        \labelwidth2.5em\itemindent3.5em\labelsep1em%
      }%
    \let\saveitem\item%
    \def\item{\saveitem%
      \def\@currentlabel{\curlabelspeicher$\,$\labelnummer}}%
    \let\savelabel\label%
    \def\label##1{\savelabel{##1}%
      \@bsphack%
        \ifmmode\else%
          \protected@write\@auxout{}%
          {\string\newlabel{##1item}{{\labelnummer}{\thepage}}}%
        \fi%
      \@esphack%
    }%
  }{\end{list}}%

\begin{document}

\title[Bernoulli Decompositions for Random Variables]
{On Bernoulli Decompositions  \\ for Random Variables, Concentration Bounds, \\ and Spectral Localization}

\author{Michael Aizenman} %
             \address[Aizenman]{     Princeton University, Departments of Mathematics and Physics,
        Princeton, NJ 08544, USA }
             \email{aizenman@princeton.edu}

\author{Fran\c cois Germinet}
\address[Germinet]{ Universit\'e de Cergy-Pontoise, Laboratoire AGM, UMR CNRS 8088,
D\'epartement de Math\'ematiques,
Site de Saint-Martin,
2 avenue Adolphe Chauvin,
95302 Cergy-Pontoise cedex, France}
 \email{germinet@math.u-cergy.fr}

\author{Abel Klein}
\address[Klein]{University of California, Irvine,
Department of Mathematics,
Irvine, CA 92697-3875,  USA}
 \email{aklein@uci.edu}   
 
\author{Simone Warzel} %
\address[Warzel]{Princeton University, Department of Mathematics,
             Princeton, NJ 08544, USA}%
\email{swarzel@princeton.edu}%

\begin{abstract} 
As was noted already by A. N. Kolmogorov,  any random variable has a Bernoulli component.  This observation  provides a tool for the extension of results which are known for Bernoulli random variables  to arbitrary distributions.     Two applications are provided here: {\em i.\/}  an anti-concentration bound for a class of functions of independent random variables, where probabilistic bounds  are extracted from combinatorial results, and  
{\em ii.\/}  a proof, based on the Bernoulli case, of spectral localization for random Schr\"odinger operators with arbitrary probability distributions  for the single site coupling constants. 
For a general random variable, the  Bernoulli component may be defined so that its conditional variance is uniformly positive.  The natural maximization problem  is an optimal transport question which is also addressed here.  
\end{abstract}

\date{June 29, 2007}

\maketitle

\tableofcontents

\newpage


\section{Introduction}

This article has a twofold purpose.  As a general observation it is noted that in any random variable one may find a Bernoulli component.  A decomposition which is based on the above observation allows then to extend 
 results which for systems of Bernoulli variables are available by combinatorial methods  to systems of random variables of arbitrary  distribution.  
 
 A Bernoulli decomposition of a real-valued random variable $X$ is a representation of  the form
\begin{equation} \label{eq:t0}
X \ \stackrel{\mathcal D}{=}  \ Y(t) \  + \ \delta (t)  \  \eta  
\, , 
\end{equation} 
where $Y(\cdot)$ and  $ \delta (\cdot) \ge 0 $ are functions  on $(0,1)$, 
the variable $t$  is uniformly distributed  in $(0,1)$, and   
$\eta$ is a Bernoulli random variable taking the values $\{ 0, 1\}$ with
probabilities $\{1-p, p\}$ independently of $t$.  The relation in \eqref{eq:t0} is to be understood as expressing equality of the distributions of the corresponding random variables. 

Bernoulli decompositions are constructed here for arbitrary random variables  of non-degenerate 
distributions.  
For certain purposes it is useful to have 
positive uniform conditional variance of the Bernoulli term, i.e.,  
\beq \inf_{t\in (0,1)} \delta(t) > 0 \,.  
\eeq    
We present such a representation below and discuss related issues of  optimality.

Two applications mentioned here: {\em i.\/} anti-concentration bounds for monotone, though not necessarily linear, functions of independent random variables, and  {\em ii.\/} a proof, based on the Bernoulli case \cite{BK}, of spectral localization for random Schr\"odinger operators with arbitrary probability distributions  for the single site coupling constants. 

In the first application, we consider  functions 
$ \Phi(X_1,\ldots, X_N)$  of independent non-degenerate random variables $\{X_j\}$ whose distributions  are either identical  or, in a sense explained below, are of widths greater than some common $b_X >0$.    
It is shown here that if for some $\varepsilon >0$ the function satisfies 
\begin{equation}\label{condFint}
\Phi(\boldsymbol{u} + v \boldsymbol{e}_j)-\Phi(\boldsymbol{u})  \  > \   \varepsilon  
  \end{equation}
for all $v \ge b_X$, all $\boldsymbol{u}  \in \R^N$,  and $j=1, \ldots,N$,   
where $ \boldsymbol{e}_j $ is the  unit vector in the $j$-direction, then the following concentration bound applies: 
    \beq \label{concestint}
\sup_{x\in \R} \, \P\left(  \set{\Phi(X_1,\ldots,X_N) \in [x, x+\varepsilon] } \right)  \ \le\ \frac{C_X}{\sqrt{N}} \,  ,
\eeq
with a constant $ C_X < \infty $ which depends on the uniform bounds on the distributions of $\{X_j\}$.  
The proof employs the Bernoulli representation along with the combinatorial  bounds of Sperner~\cite{S}, and the more general LYM lemma~\cite{E}.  

The use of combinatorial estimates for concentration bounds first appeared in the context of Bernoulli variables in 
P. Erd\"os'  variant of the Littlewood-Offord Lemma~\cite{Er}.  
The presence of a Bernoulli component in any random variable was noted  implicitly in the work of A. N.  Kolmogorov~\cite{Kol58}  where it was put to use in an improvement of the earlier concentration bounds of W. Doeblin and P. L\'evy~\cite{DL36,Doe39}  on linear functions of independent random variables.    
Initially, Kolmogorov did not extract the maximal benefit from the method by not connecting it with 
Sperner theory, and in particular the concentration bound in \cite{Kol58} includes an  unnecessary logarithmic factor; the corresponding improvement was made by B. A. Rogozin~\cite{Rog61}.  
The bounds  were further improved in a series of works, in particular \cite{Ess68,Kes69,Rog73} 
where use was also made of other methods.  
One may  note  here that 
perhaps quite naturally a general method like the Bernoulli decomposition is not optimized for specific applications.   
Nevertheless, it has the benefit of providing a simple perspective on a number of topics.

In our second application, we establish spectral localization 
for a broad class of continuum,  alloy-type random Schr\"odinger operators (cf. \eqref{Hom}), building on a result of J. Bourgain and C. Kenig~\cite{BK} for the Bernoulli case.   
The model and the results are  presented more explicitly in Section~\ref{sec:loc}.  
The main point to be made here is that the understanding of  spectral localization for the Bernoulli case can be extended through the Bernoulli decomposition to random operators with single site coupling parameters of arbitrary distribution (cf. Theorem~\ref{thm:gen_loc}).

\section{Bernoulli decompositions for random variables}
\label{sec:rep}

Randomness often is in the eyes of the beholder, as probability
measures are used to express averages over specified sets of rather
varied nature.  However, it may be true that the  most elementary
model underlying the basic popular perception of probability is the simple
 `coin toss', with two possible outcomes:  heads or tails, which  is modeled by a Bernoulli random variable:   a binary variable   equal to  $1$ with probability $p$ and equal  to $0$ with probability $1-p$.

\subsection{The decomposition in two variants} 
The following statement assert that any real
valued random variable has   a Bernoulli component, which can even
be chosen to be of uniformly positive variance. 

Given a real random variable  $X$ by default we shall 
denote its  probability distribution by $\mu $ 
and let $G: (0,1) \to (-\infty,\infty) $  be the function defined by
\begin{equation} \label{eq:G}
G(t) \ := \ \inf \{ u\in \R \, : \, \mu\left( (-\infty, u] \right)
\ge t \}  \, . 
\end{equation}
One may observe that $G $ is the `inverse' distribution function 
of $\mu$, which takes values in the essential range of $X$. 
It can alternatively be described by  
\beq  \label{eq:G2}
G(t)\le u \quad  \Longleftrightarrow \quad \mu \left( (-\infty, u]\right)  
\ge t \, ,
\eeq  
and satisfies
$ \mu \left( (-\infty, G(t) - \eps]\right) < t \le \mu \left( (-\infty, G(t) ]\right) $ for all $t \in \R$ and $\eps>0$.


\begin{theorem} \label{thm:rep_t+}
Let $X$ be a non-degenerate real-valued random variable with a 
probability distribution $\mu $.
Then, for each $p\in (0,1)$,  $X$ admits a decomposition of
the form:
\begin{equation} \label{eq:t+}
X \ \stackrel{\mathcal D}{=}  \ Y_{p}(t) \  + \ \delta_{p}^+ (t)  \  \eta  
\, , 
\end{equation} 
in the sense of equality of the corresponding probability
distributions, where:
    \begin{enumerate}
\item   $\eta$ and $t$ are independent random variables, with
$\eta$ a binary variable taking the values $\{ 0, 1\}$ with
probabilities $\{1-p, p\}$, correspondingly, and
$t$ having   the  uniform distribution  in $(0,1)$,
\item
$Y_{p}: \, (0,1) \mapsto (-\infty ,\infty)$  
is the monotone non-decreasing function  
\beq 
Y_{p}(t) \   := \   G( (1-p) t ) \, ,  
\eeq    
     \item  
$\delta_{p}^+: \, (0,1) \mapsto [0,\infty)$  is the  function  
\beq \label{eq:ZZZ}
\delta_{p}^+(t) \  := \   G(1-p+pt) - G((1-p)t )  \, ,  
\eeq
\item for  at least one value of $ p \in (0,1)$ we have
   \begin{equation}\label{infdelta}
\beta^+(p,\mu)  \  :=  \inf_{t\in (0,1)} \delta_{p}^+(t) > 0\, .
\end{equation}
\end{enumerate}
\end{theorem}

Some explicit expressions for $\beta^+(p,\mu)$ are mentioned in Remark~\ref{rem1} below.   The Bernoulli component of the measure is not a uniquely defined notion, 
and other representations similar to \eqref{eq:t+} but 
with different 
distributions for the conditional variance of the Bernoulli component, 
i.e.,  for $\delta(t)$,  can also be obtained.   In the following construction 
its uniform positivity  may be lost but one gains the 
feature that the range of values which  $\delta$ assumes   
reaches up to the diameter of the support of the measure $\mu$. 

\begin{theorem} \label{thm:rep_t1}
Let $X$ be a non-degenerate real-valued random variable with 
probability distribution $\mu $.
Then, for each $p\in (0,1)$, $X$ admits a decomposition
of the form:
\begin{equation} \label{eq:t}
X \ \stackrel{\mathcal D}{=}  \ Y_p(t) \  + \ \delta_{p}^-(t)  \  \eta
\end{equation}
where $t$, $\eta$ and the function $Y_{p}$ are as in 
Theorem~\ref{thm:rep_t+},   satisfying the above (1) and (2), but instead of 
(3) and (4)  the following holds   
    \begin{enumerate}
  \item[(3')]  $\delta_{p}^-: \,  (0,1) \mapsto [0,\infty)$  is the non-increasing function:
\beq \label{eq:del}
\delta_{p}^-(t) \  :=   G(1-pt) - G((1-p)t )\, ,
\eeq
\item[(4')]  for any $ x_- <x_+$ and  $p_\pm >0$  such that
\begin{equation}\label{condp1}
\P \left( \set{X  \le x_- }\right )\ge     p_-  \quad\mbox{and}\quad   
\P \left( \set{X  > x_+} \right )   >  p_+   \,   , 
\end{equation}
at the particular value  $p=\frac{p_+}{ p_{-} + p_{+}} $ we have 
\beq  \label{eq:JB} 
\P_t \left( \set{\delta_{p}^-(t)    >    x_{+} -  x_{-} }\right ) \ \ge\ p_{-} + p_{+} \, , 
\eeq 
where the probability is with respect to the uniform random variable $t$.
\end{enumerate}
\end{theorem}

In the proofs we employ two versions of  what is called here the
\emph{Pac-Man} algorithm for the construction of a joint distribution
$\rho $ of a pair of random variables, of the form $\{Y_1(t),
Y_2(t)\}$, whose marginal probability measures, 
$ \rho_1 $, $ \rho_2$, satisfy  
\begin{equation}  \label{eq:margin}
\mu \ = \  (1-p) \, \rho_1 \ +\  p \, \rho_2  \, .
\end{equation}
The representations \eqref{eq:t+} and \eqref{eq:t} correspond to letting:
\begin{eqnarray} \label{eq:delta}
     Y_p(t)  & := &  Y_1(t)  \nonumber  \\
      \delta_{p}^{\pm}(t)  & := &  Y_2(t) - Y_1(t)  \, .
     \end{eqnarray}
The two Theorems will be proven in reverse order. 

\begin{proof}[Proof of Theorem~\ref{thm:rep_t1}:]
We start by recalling the known observation that for any continuous
function $\phi \in C(\R)$:
\begin{equation} \label{eq:G-mu}
     \int_0^1 \phi(G(s)) \, ds \  = \   \int_{\R} \phi(x) d\mu(x)  \  \, .
     \end{equation}
This relation allows to represent, in terms similar to
\eqref{eq:t},  as
\begin{equation}  \label{eq:XG}
X \  \stackrel{\mathcal D}{=} \ G (t) \, , 
\end{equation}
with  $t$ the random variable with the uniform  distribution in $[0,1]$.

Extending the above representation, we now define
a pair of coupled random variables through the following functions of
$t \in [0,1]$:
\begin{eqnarray} \label{eq:defY}
Y_1(t) &:=& G\left((1-p)t\right) \nonumber \\
Y_2(t) &:=& G\left( 1-p t\right)
\end{eqnarray}
As is the case of $G$ in \eqref{eq:XG}, the functions $Y_1$ and $ Y_2
$ are made into random variables by assigning to them the joint
probability distribution which is induced by Lebesgue measure on
$[0,1]$.
Their marginal distributions satisfy \eqref{eq:margin}, since  for
any continuous  function $\phi \in C(\R)$ 
    \begin{eqnarray}  \label{eq:marginal} 
    (1-p) \int_0^1 \phi(Y_1(t)) \, dt \ &+& \  p \int_0^1 \phi(Y_2(t))
\, dt    \nonumber \\
     & = & (1-p) \int_0^1 \phi(G((1-p)t)) \, dt \ + \  p \int_0^1
\phi(G(1-pt) )\, dt   \nonumber \\
     &=&  \left[ \int_0^{1-p} + \int_{1-p}^1 \right]  \phi(G(s)) \, ds \
      = \   \int \phi(x) d\mu(x)
    \end{eqnarray}  
where the last equality is by  \eqref{eq:G-mu}.

By \eqref{eq:marginal} the random variable seen on the right side of
\eqref{eq:t} has the same distribution as $X$.
The statement (2) readily follows from the definition \eqref{eq:defY}
and \eqref{eq:delta}. 

For a proof of (4') we note that \eqref{condp1} is equivalent to 
\beq
G(p_-) \ \le \ x_- \qquad  \mbox{ and} \qquad   G(1-p_+) \ > \ x_+ \, .  
\eeq 
This implies  $\delta_p^+(t) > x_+ - x_-$ for all $t\le p_+ + p_-$, and 
hence \eqref{eq:JB} holds true. 
\end{proof}

In the above proof, one may regard the functions $Y_1(t)$ and $Y_2(t)
$  defined by   \eqref{eq:defY}
as describing  the motion of a pair of markers which move along $\R$
consuming the $\mu$-measure at the steady rates of $(1-p)$ and $p$,
correspondingly.  The markers leap discontinuously over intervals of
zero $\mu$-measure and, conversely, linger at points of positive
mass.  Their motion invokes the image of a linear version of the
\emph{Pac-Man} game, and hence we shall refer to the construction by
this name.
Whereas in the above construction the Pac-Men move towards each
other, we shall next use the Pac-Man algorithm with one marker
chasing the other.

\begin{proof}[Proof of Theorem~\ref{thm:rep_t+}]
For the representation \eqref{eq:t+} we shall employ the following
variant of \eqref{eq:defY}:
\begin{eqnarray} \label{eq:defY+}
Y_1(t) &:=& G\left((1-p)t\right) \nonumber \\
Y_2(t) &:=& G\left( 1-p  +pt\right)
\end{eqnarray}
In this case, both
$ Y_1$  and $ Y_2 $ are monotone non-decreasing in $t$ and
\begin{equation}
Y_1(t) \le G(1-p) \le G(1-p +0) \leq Y_2(t)
\end{equation}
for all $ t \in (0,1)$, where $ G(1-p +0) = \lim_{\varepsilon \downarrow 0} G(1-p + \varepsilon) $. 
Moreover, for any $T \in (0,1)$ we have the lower bound
\beq \label{betaT}
\beta^+(p,\mu) \ge \min\set{ G(1-p) - Y_1(T), Y_2(T)- G(1-p+0) },
 \eeq
since
\begin{equation}\label{lowerbounddelta}
 \delta_{p}^{+}(t) \ge \begin{cases}  G(1-p) - Y_1(T) & \text{if $ 0< t \le T$},\\
 Y_2(T)- G(1-p+0) & \text{if $ T\le t < 1$} \, .
 \end{cases}
\end{equation}

For a sufficient condition for the uniform positivity of  
$ \delta_{p}^{+}(t) = Y_2(t) - Y_1(t) $  let us consider the
arrival/departure times:
\begin{eqnarray}
T_1  \ &= \  \inf \{ t\in (0,1) \, : \, Y_1(t) = G(1-p) \}  \qquad &
\mbox{(arrival time of $Y_1$)}  \, , \nonumber \\
T_2  \ &= \  \sup \{ t\in (0,1) \, : \, Y_2(t) = G(1-p + 0 ) \}  \qquad &
\mbox{(departure time of $Y_2$)} \,  .   \nonumber 
\end{eqnarray}
The times
$T_1, T_2$ are non-random and depend on $ p $ and $ \mu $ only. If
\begin{equation} \label{eq:TT}
T_1  > T_2  \, ,
\end{equation}
then for each  $ T \in (T_2, T_1) $ we have  
\beq  \label{lowerbounddelta2}
\beta^+(p,\mu) \ge \min\set{ G(1-p) - Y_1(T), Y_2(T)- G(1-p+0) } > 0.
\eeq 
The collection of $p \in (0,1) $ such that \eqref{eq:TT} is not empty whenever the support of the
measure includes more than one point. 
    \end{proof}

\begin{remark} \label{rem1} 
\begin{nummer}

\item  {\em Explicit lower bounds on $\beta^+$.\/}  For the Bernoulli
decomposition which is presented in Theorem~\ref{thm:rep_t+}
(i.e., based on the `chasing Pac-Men' algorithm), an expression for the lower bound $\beta^+(p,\mu)$  in  terms of the distribution function
of $\mu$ is given in \eqref{eq:beta_F} below.
    A simple  lower bound  can be obtained in terms of just the
``half-time'' points for the two markers. i.e., from \eqref{betaT} with $T=\frac 1 2$:
\begin{equation}
\beta^+(p,\mu) \ \ge  \ \min\left \{ \left[
G(1-p)-G\left(\frac{1-p}{2}\right) \right], \, \left[ G\left(
\frac{2-p}{2}\right) - G(1-p) \right] \right \}       \, .
\end{equation}
This  shows that for continuous measures $ \mu $ one has
$\beta^+(p,\mu)>0$, i.e., \eqref{infdelta},  for any $p\in(0,1)$.

  If the support of $\mu$ consists of exactly two points the representation
\eqref{eq:t} is trivially available, though at a unique value of
$p\in (0,1)$. If the support of $\mu$ contains more than two points,  there exists at least one $\hat{x} \in \R$  such that
\beq \begin{split}\label{xhat}
&  \mu \left( (-\infty, x]\right) <  \mu \left( (-\infty, \hat{x}]\right) \quad \text{if} \quad  x <\hat{x} \, , \\
&  0 <  \mu \left( (-\infty, \hat{x})\right)\le  \mu \left( (-\infty, \hat{x}]\right)<1 \, .
\end{split}\eeq  
At the particular value  $p= 1 -   \mu \left( (-\infty, \hat{x})\right)$ we then have $G(1-p)=\hat{x}$ and 
\begin{equation}\begin{split}\label{infdeltaexp}
\beta^+(p,\mu)   \ge   &  \min \set{ \hat{x}-  G( (1-p) t ),   G(1-p+pt)- \hat{x}}     > 0 \\
&\text{for each $t$   such that}\quad  \frac  { \mu\pa{\set{\hat{x}} }}{p} < t <  1 \,  .
\end{split}\eeq

\item   {\em An alternative form.\/}  
For another form of a Bernoulli
decomposition, with a binary random variable  $\sigma = \pm1$,
let
\beq \label{eq:sigma}
\sigma \  =\  2\eta -1 \qquad \text{and}   \qquad  W \ = \ Y_{p} +\tfrac 1 2 \delta_{p}^{+} \,   .
\eeq
When such a substitution is made in \eqref{eq:t+} the two resulting
functions $W(t)$ and $\delta_{p}^{+}(t)$  are monotone non-decreasing in $t$
and  $\delta_{p}^{+}(\cdot)$ is constant over each interval of constancy
of $W(\cdot)$.  It follows that the value of $\delta_{p}^{+}(t)$ can be
expressed in terms of $W(t)$, and thus one obtains a
representation of the form:
\beq
X\  \stackrel{\mathcal D}{=} \  W +b(W) \sigma \, ,
\eeq
with $W$ and $\sigma$ independent random variables, and $b(\cdot)$ a
measurable function which is determined by $\mu$ and $p$.
  
\item {\em Some precedents.\/}  
As was commented above, the Bernoulli decomposition of Theorem~\ref{thm:rep_t1} with $ p = 1/2 $ 
has appeared already in a work of A.\ N.\ Kolmogorov \cite{Kol58}.
For random variables with values in $\Z$, the representation
\eqref{eq:t+} of  Theorem~\ref{thm:rep_t+}
    is related to the somewhat similar representation (though with
$\delta = 0,1$, not necessarily positive) which D. McDonald showed to
be useful for  the analysis of local limit theorems for integer
random variables (\cite{McD}).
\end{nummer}
\end{remark}

\subsection{Optimality of the Pac-Man algorithm}
\label{sect:optimality}

In applications of the decomposition it is desirable to maximize the conditional variance of the binary term.  We shall now address related questions from  an \emph{optimal transport} perspective, and in
particular establish optimality, in a certain limited sense, of the
`chasing Pac-Men' construction.

In addition to the explicit choices presented in
Theorems~\ref{thm:rep_t+} and \ref{thm:rep_t1} there are other
possibilities for a Bernoulli decomposition  of the form
\eqref{eq:t+}.   
With a change of variables as in  \eqref{eq:delta}, 
such representations can alternatively be expressed  
in terms of joint distributions of the variables 
$Y_1$, $Y_2 $ with the properties listed in the following definition.  

\begin{definition}
A $(1-p,p)$ Bernoulli decomposition of a  probability measure $\mu$ on
$\R$ is a probability measure $\rho(dY_1\, dY_2)$ on $\R^2$ whose
marginals $\rho_1$ and $\rho_2$ satisfy:
\begin{equation}  \label{eq:marginals}
(1-p) \, \rho_1  \ + \  p \,  \rho_2  \ = \  \mu   \, .
\end{equation}

This concept can of course be easily generalized to variables with
values in $\R^d$, or $\C$.  For real variables the defining condition
\eqref{eq:marginals} is conveniently expressed in terms of the
distribution functions, as
\begin{equation}  \label{eq:margins_F}
(1-p) \, F_1(x) \ + \ p \, F_2(x) \ = \ F(x)
\end{equation}
where
$ F(x)  =   \mu( (-\infty,x] ) $, and $
F_j (x)  =   \rho( \{ Y_j \le  x \} ) $ for $j=1,2$.

For each Bernoulli decomposition of a probability measure on $\R$ we denote:
\beq
\left\{ \begin{array}{c} \beta^\# \\ \beta_* \end{array} \right\} \, (p,\rho) \ := \  \rm{ess}_\rho \left\{ \begin{array}{c}
\rm{sup} \\ \rm{inf} \end{array} \right\} \, (Y_2 - Y_1) \,  .
    \eeq
\end{definition}

\begin{theorem} \label{thm:comparison}
For any $(1-p,p)$, among all the Bernoulli decomposition of a given
probability measure  $\mu$ on $\R$:
\begin{enumerate}
\item
The minimal conditional  variation $ \beta_*(p,\rho) $ is maximized
by the `chasing Pac-Men' algorithm which is presented in the proof of
Theorem~\ref{thm:rep_t+}, i.e. for any Bernoulli decomposition 
\begin{equation} \label{eq:betamax}
\beta_*(p,\rho)  \ \le \  \beta^+(p,\mu) := 
{\essinf}_{t\in (0,1)} \delta_{p}^+(t)   \, , 
\end{equation} 
where ${\essinf}_{t\in (0,1)} $ yields the same value as $\rm{inf}_{t\in (0,1)} $.  
\item
The maximal conditional  variation $ \beta^\#(p,\rho) $ is maximized
by the `colliding Pac-Men' algorithm of Theorem~\ref{thm:rep_t1}, for
which  $\beta^\#(p,\rho) $ equals the diameter of the essential
support of $\mu$.
\end{enumerate}
\end{theorem}

The equality:  $ \rm{essinf}_{t\in [0,1]} \delta_{p}^+(t)= \rm{inf}_{t\in [0,1]} \delta_{p}^+(t)$ is a simple consequence of the left-continuity property of the chasing Pac-Men algorithm, where $Y_j(t) = Y_j(t-0)$ and hence also $\delta^+(t) =   \delta^+(t-0)$.  

To prove \eqref{eq:betamax}  let us first establish a helpful expression for
$\beta^+(p,\mu)$.  
Denoting by   
$F^+_j  $ the distribution functions corresponding to $Y_1$
and $Y_2$ of  \eqref{eq:defY+} we have:

\begin{lemma}  \label{lem:Pac-Man}
For the `chasing Pac-Men' construction, of Theorem~\ref{thm:rep_t+}:
\begin{equation} \label{eq:F*}
F^+_1(x)  =  \frac{1}{1-p} \min\left\{ F(x),\, 1-p \right \} \, ,  \quad
F^+_2(x)  =  \frac{1}{p} \max\left\{ F(x)+ p -1 ,\, 0 \right \}  \,
\end{equation}
and
\begin{equation} \label{eq:beta_F}
\beta^+(p,\mu)\ = \ \sup \left \{ b\in \R \, : \, F^+_1(x) \ge
F^+_2(x+b) \quad \mbox{for all} \; x\in \R  \right \} \, .
\end{equation}
\end{lemma}

\begin{proof}
The statements \eqref{eq:F*} follow directly from the definition of
the Pac-Man process \eqref{eq:defY+}.   
In the derivation of \eqref{eq:beta_F}, we shall use the fact that for all 
$t\in(0,1)$ and 
$\eps >0$: 
\beq 
F_j^+( Y_j^+(t) - \eps)   <   t \ \le  F_j^+( Yj^+(t)) 
\eeq 

Let $ S := \sup \left \{ b\in \R \, : \, F^+_1(x) \ge
F^+_2(x+b) \quad \mbox{for all} \; x\in \R  \right \}$. 
Clearly, for any $u > S$,  there is $x\in \R$ such that 
\beq 
F_1(x) \  < \ F_2(x+u) \, .
\eeq 
It follows that for any $t \in (F_1(x), F_2(x+u)) $: 
\beq 
Y_2(t) \ \le \ x+u  \qquad \mbox{and} \qquad Y_1(t) \  > \ x \, , 
\eeq 
 and therefore $\delta^+(t) = Y_2(t) -  Y_1(t)  \ \le \ u$.  Thus: 
 $\rm{inf}_{t\in (0,1)} \delta_{p}^+(t) \le S$. 
 
For the converse direction, let us  note that due to the monotonicity of 
$F$  the condition on $b$ in \eqref{eq:beta_F} is satisfied by all 
$u<S$.  Thus, if  $u < S$, 
then,   for all $x\in \R$:
\beq 
 F_2^+(x+u) \ \le \  F_1^+(x-0)  \, , 
\eeq 
and hence for any $t\in (0,1)$:
 \beq 
F_2^+(Y_1^+(t)+u) \ \le \  F_1^+(Y_1^+(t)-0)  \le \  t  \, , 
\eeq  
which implies that  $Y_1^+(t)+u \le Y_2^+(t) $.  Therefore
\beq 
\inf_{t\in(0,1)}\pa{ Y_2^+(t)  - Y_1^+(t)}  \ \ge \ u \, .
\eeq 
It follows that   $\rm{inf}_{t\in [0,1]} \delta_{p}^+(t) \ge S$, 
which completes the proof of \eqref{eq:beta_F}.
\end{proof}

\begin{proof}[Proof of Theorem~\ref{thm:comparison} ]
The second assertion is an elementary consequence of \eqref{eq:del}.  
To prove (1) we shall show that for any
$b >\beta^+(p,\mu)$ it is also true that
$ b>\beta_*(p,\rho) $.\\

The condition \eqref{eq:margins_F} readily
implies that $(1-p)\, F_1(u) \le F(u)$, or
$ F_1(x) \  \le\   \min \left \{ (1-p)^{-1} F(x), \, 1\right \} $,
and hence
\begin{eqnarray} \label{eq:F-F*}
F_1(x)  &\le &  F^+_1(x) \  \nonumber \\[1ex]
F_2(x) &  \ge &   F^+_2(x)  \,  .
\end{eqnarray}
Now, by Lemma~\ref{lem:Pac-Man}, for any $b > \beta^+(p,\mu)$ there
exist some $t, u\in \R$, such that
$ F^+_1(u) \ = t \ < \  F^+_2(u+b) $
and therefore, due to \eqref{eq:F-F*}, also
\begin{equation}\label{eq:FTF}
    \ F_1(u) \ \le t \ < \  F_2(u+b) \, .
\end{equation}
Eq.~\eqref{eq:FTF} means that
 $ \rho \left(  \{ Y_1   \le  u  \} \right)  \le  t  $  and  $ \rho \left(
\{ Y_2  >  u+b  \}  \right)  <  1- t  $.
Since the probabilities of the two events add to less than $1$  the
complement of their union is of positive probability, and this
implies:
\begin{equation}
\rho \left(  \{ Y_2 - Y_1  \le  b  \} \right) \ > \ 0 \, ,
\end{equation}
and hence $ b>\beta_*(p,\rho) $.  This concludes the proof of \eqref{eq:betamax}.
\end{proof}

\begin{remark}
    The idea of seeking optimal joint realizations of random variables
with constrained marginals has allowed to present a wide range of
analytical results from a common `optimal transport' perspective
(see, e.g., \cite{V}).   The most familiar variants of the problem
concern couplings which {\em minimize} a distance function between the
two coupled variables.   As our discussion demonstrates, it may also
be of interest to seek couplings which {\em maximize} the difference
between the two variables with constrained marginals.
\end{remark}


\section{Concentration Bounds}\label{sec:Sperner}

We shall now demonstrate how the Bernoulli decomposition  yields probabilistic bounds from combinatorial results.   
If there is any novelty in this section it is in the formulation of the bounds for the non-linear case, as the two main ideas were noted before in the context of linear functions:  
P. Erd\"{o}s~\cite{Er} observed that concentration bounds for linear functions of Bernoulli variables can be derived from the combinatorial theory of E. Sperner~\cite{S}, 
and B. A. Rogozin \cite{Rog61} has used the Bernoulli decomposition of A. N. Kolmogorov \cite{Kol58} for the  further extension of these bounds to arbitrary random variables.  

First,  we present some essentially known results of Sperner theory; 
in the second subsection these results  
will be combined with the Bernoulli decomposition
to yield  concentration bounds for functions of independent
random variables.

\subsection{Probabilistic Sperner Estimates}

The configuration space $\{0,1\}^N$  for a collection of Bernoulli random variables
$\bdeta=\{ \eta_1, ..., \eta_N\}$  is   partially
ordered by the relation defined by:
\begin{equation}  \label{eq:po}
\bdeta  \prec \bdeta' \quad  \Longleftrightarrow  \qquad \ \mbox{for all\
\ } i \in \{1,..., N\}:  \quad \eta_i \le \eta_i' \, .
\end{equation}
A set ${\mathcal A} \subset \{0,1\}^N$ is said to be an \emph {antichain}
if it does not contain any pair of configurations which are
compatible in the sense of ``$\prec $''.  
The original  Sperner  Lemma states that  for any such set:  
$ \abs{\cA} \le \binom{N}{[\frac N 2]} $.  A more general result is the 
 LYM inequality for antichains (cf. \cite{A}):
\beq \label{LYM}
\sum_{\bdeta \in \cA}  \frac 1 {\binom{N}{\abs{\bdeta}}} \ \le \ 1 \, ,
\eeq
where $\abs{\bdeta} = \sum \eta_j $.

The LYM inequality has the following probabilistic implication.

\begin{lemma}\label{lemprobsp}  Let  $\{
\eta_j \}$ be independent copies of a Bernoulli random variable
$\eta$ with
    \beq\label{Bernp}
    \P\left( \set{\eta=1} \right)  = p  \, , \qquad    \P\left(
\set{\eta=0}\right) =  q  := 1-p \, ,
     \eeq
where  $p\, \in (0,1)$. Then for any antichain $\cA\subset \set{0,1}^N$:
    \beq  \label{probSpernerpq}
\P\left( \{\bdeta \in  \cA\} \right) \  \le \  \frac {\Theta}
{\sigma_\eta\sqrt{N} },
\eeq
where  $\bdeta=(\eta_{1},\ldots,\eta_{N})$, $\sigma_\eta = \sqrt{pq}$   
is the standard deviation of $\eta$,  and
$\Theta$ is an independent constant which does not exceed  $2\sqrt{2}$.
\end{lemma}

\begin{proof}
Let $A_{k}$ be the subset of $\cA$ consisting of configurations with
$\abs{\bdeta}=k$.  Then:
\begin{align}  \label{preStirling}
\P\left(\{ \bdeta \in  \cA\} \right) = \sum_{k=0}^{N} p^{k}q^{N-k}\abs{\cA_k}=
    \sum_{k=0}^{N} b(k;N,p)\frac {\abs{\cA_k}}{\binom{N}{k}}
    \le  \max_{k=0,1,\ldots,N}b(k;N,p),
\end{align}
where $b(k;N,p):= p^{k}q^{N-k}\binom{N}{k}$ is the binomial
distribution, and the inequality is by \eqref{LYM}.    
The  maximum of $b(k;N,p)$  over $k$, which is known to occur near
$k=pN$ (cf. \cite[Theorem~1 on p. 140]{F}) yields
\eqref{probSpernerpq}. 
\end{proof}

The bound \eqref{probSpernerpq}   has the virtue of being valid for 
all $N$; for $N\to \infty $ it holds with a smaller constant which 
tends to  the asymptotic value  $\Theta \to 1/ \sqrt{2\pi }$ (implied 
by \eqref{preStirling} and Stirling's formula).

Following is an extension of Lemma~\ref{lemprobsp} to the case of
non-identically distributed random variables.

\begin{lemma}  \label{lem:extsp}  Let
$\bdeta=(\eta_{1},\ldots,\eta_{N})$, where $\{ \eta_j\}$ are
independent Bernoulli random variables with possibly different values
of $p_j$, and set
\beq \label{eq:alpha}
\alpha:= \min_{j=1,2,\ldots,N} \, \min\left\{ p_{j}, 1 -   p_{j}\right \} \in (0,1/2]\, .
\eeq
Then, for any antichain $\cA\subset \set{0,1}^N$:
    \beq  \label{Sper_varied}
\P\{\bdeta \in  \cA\} \le \frac {\widetilde \Theta} { \alpha \sqrt{ N} } \, ,
\eeq
where $\widetilde \Theta$  is an independent constant 
which does not exceed  $4$.
\end{lemma}

The proof gives us the chance to introduce the technique  of `double
sampling'.

\begin{proof}
We start from the observation that any
Bernoulli variable $ \eta $ with parameter $p_\eta$  as in
\eqref{Bernp}  may be decomposed  in terms of two independent
Bernoulli variables $ \chi $ and $ \xi $
as
\begin{equation}\label{eq:resampling}
     \eta \ \stackrel{\mathcal D}{=} \ \xi \,  \chi \, ,
\end{equation}
with $ p_\xi\, p_\chi = p_\eta $.

By the definition of $ \alpha $, eq.~\eq{eq:alpha},  $ p_j\in[\alpha, 1-\alpha] $ 
for all $j=1,2,\ldots,N$.
Hence the variables $ \bdeta $ may be represented as in \eqref{eq:resampling} with 
independent identically distributed (iid) 
Bernoulli variables $ \{\chi_j \} $   with common $ p_\chi := 1-\alpha  $.  
We  abbreviate this representation as
$ \boldsymbol{\xi} \,   \boldsymbol{\chi} := (\xi_1  \chi_1, \ldots,
\xi_N \chi_N) $.  
Evaluating the probability by first conditioning on the values of $ \boldsymbol{\xi} $, 
one has
\begin{equation}
    \P\{\bdeta \in  \cA\}  =  \E \left[ \P\left\{\boldsymbol{\xi} \,
\boldsymbol{\chi} \in \cA \, | \,  \boldsymbol{\xi} \right\} \right]
\end{equation}
For specified values of the variables $ \boldsymbol{\chi}$ , the
event $\cA$ depends only on the values of $\chi_j$ with $j$ in the set
$ J_{\boldsymbol{\xi}} := \{ j \, : \, \xi_j \neq 0 \} $, and as such
it is an antichain in $ \{0,1\}^{J_{\boldsymbol{\xi}}}$.
Bounding its conditional probability  by
Lemma~\ref{lemprobsp} we obtain
\begin{equation}\label{eq:partialsperner}
\P\left\{\boldsymbol{\xi} \,   \boldsymbol{\chi} \in \cA \, | \,
\boldsymbol{\xi} \right\}
\leq \min\left\{ 1, \frac{\Theta}{\sigma_\chi \sqrt{|
J_{\boldsymbol{\xi}}}|} \right\} \, ,
\end{equation}
where $ \sigma_\chi = \sqrt{\alpha (1-\alpha)} $ is the common
standard deviation of $\chi_j$.

To conclude the proof of \eqref{Sper_varied} it remains to estimate
the expected value of the right hand side of
 \eqref{eq:partialsperner}, where
$ | J_{\boldsymbol{\xi}}| = \sum_{j=1}^N \xi_j$.
Noting that $\E (\xi_j ) = p_{\xi_j} =p_j/(1-\alpha) \geq \alpha/(1-\alpha) $, we
see that the mean satisfies:
\beq
\E (| J_{\boldsymbol{\xi}}|) \ \ge \ \frac{\alpha N}{1-\alpha} \, .
\eeq
The event $\left\{  | J_{\boldsymbol{\xi}}| \ \le \ \alpha N /2(1-\alpha)
\right\}$ is of exponentially small probability, as can be seen by a standard
large deviation estimate for independent variables.  It then readily follows that
\beq
\E\left( \min\left\{ 1, \frac{\Theta}{\sigma_\chi \sqrt{|
J_{\boldsymbol{\xi}}}|} \right\} \right) \ \le \
\frac{\widetilde\Theta}{ \alpha \sqrt{N}}  \, , 
\eeq 
with a constant for which elementary estimates yield ${\widetilde\Theta} \le 4 $.
\end{proof}

\begin{remark}\label{rem:multiset}
The above notions and results have natural extensions to 
integer valued independent random variables, $\btau=(\tau_1,\tau_2,\ldots,\tau_N)$, whose
configuration space, $\Z^N$, is also partially ordered by the natural
extension of the relation \eqref{eq:po}.  
The Bernoulli  decomposition \eqref{eq:t} can be used for 
an extension of the probabilistic bound of Lemma~\ref{lem:extsp} to this more general case.   
One way to derive the general statement is through the application of the bound \eqref{Sper_varied} 
to the conditional probability for the Bernoulli component, as in the arguments which appear below.  Alternatively, 
one may note that the statement directly follows from Theorem~\ref{thmconc} which is presented in the next section.  
     
For completeness it should be added that in addition to the anti-concentration upper bounds it is of interest to know the asymptotic 
behavior.  That is covered by known results, such as is presented  in Engel \cite[Theorem~7.2.1]{E}:
\beq
\lim_{N\to \infty}  {\sigma_\mu \sqrt{2 \pi N}} \set{\max_{\cA\subset
\set{0,1,\ldots,k}^N \; \text{antichain}} \P\{\btau \in  \cA\}}=1 \, , 
\eeq
which amounts to a `local' central limit theorem (CLT).    
\end{remark}
   
\subsection{Concentration Bounds for Functions of Independent Random Variables}

We shall now employ the Bernoulli decomposition of
Section~\ref{sec:rep}, along with
the results presented in the previous subsection,
for  an upper bound on the concentration probability 
\begin{equation}
  Q_Z(\xi) := \sup_{x\in \R} \, \P\left(  \set{Z \in [x, x+\xi] } \right) 
\end{equation}
for random variables of the form
\beq \label{eq:Z}
    Z\ = \  \Phi(X_1,X_2,\ldots,X_N) \, ,
\eeq
where $\set{X_j}$ are independent random variables. 
 
\begin{theorem}\label{thmconc}
Let $\boldsymbol{X} = (X_1, \ldots , X_N ) $  be a collection of
independent random variables  whose distributions  satisfy, for all $j\in \{1,...,N\}$:
\begin{equation}\label{condp}
\P \left( \set{X_j \le x_- }\right )   \ge   p_-  \quad\mbox{and}\quad  
\P \left( \set{X_j  > x_+} \right )   >  p_+   \, 
\end{equation}  
at some $p_\pm >0$  and $ x_- <x_+$, 
and  $\Phi: \R^N \mapsto \R$ a function such that 
for some $\varepsilon >0$ 
\begin{equation}\label{condF}
\Phi(\boldsymbol{u} + v \boldsymbol{e}_j)-\Phi(\boldsymbol{u})  \  > \   \varepsilon
  \end{equation}
for all $v\ge  x_+ -  x_-$, all $\boldsymbol{u}  \in \R^N$, and $j=1, \ldots,N$,   
with $ \boldsymbol{e}_j $ the  unit vector in the $j$-direction.        
Then, the random variable $Z$ which is defined by \eqref{eq:Z}   
obeys the concentration bound 
    \beq \label{concest}
Q_Z(\varepsilon) \ \le\ \frac{4}{\sqrt{N}} \,
\sqrt{\frac{1}{p_+} + \frac{1}{p_-}}   \, ,
\eeq
where $4$ can also be replaced by the constant  $\widetilde \Theta  $ of 
\eqref{Sper_varied}. 
   \end{theorem}

\begin{proof}
    We start by selecting $p\in (0,1)$ by the condition $ p = \frac {p_+}{p_+  +  p_-} $.
 Next, we represent the variables
$\{X_j\}$  using Theorem~\ref{thm:rep_t1}:
\begin{equation}
     \boldsymbol{X} \ \stackrel{\mathcal{D}}{=} \
\boldsymbol{Y}(\boldsymbol{t}) +
\boldsymbol{\delta}(\boldsymbol{t})\,  \boldsymbol{\eta} :=
     \left( Y_{p,1}(t_1) + \delta_{p,1}^-(t_1) \, \eta_1 , \ldots ,  Y_{p,N}(t_N) +
\delta_{p,N}^-(t_N) \, \eta_N \right) \, ,
\end{equation}
with   $ \boldsymbol{\eta} = (\eta_1 , \ldots , \eta_N ) $ a
collection of iid Bernoulli variables
taking values $ \{ 0, 1 \} $
with probability $ \{1-p, p \} $.  From \eqref{eq:JB} one may conclude that for all $ j \in \{ 1, \ldots , N \} $:  
\begin{equation}\label{eq:delarge}
     \P_t(\left\{\delta_{p,j}^-(t) \ge x_+ - x_-  \right\})  \geq  p_+ + p_- \, .
\end{equation}

We express the probability of the event $\{Z \in [x , x + \varepsilon]\}$
through first conditioning on the $\{t_j\}$ variables.  For all $x\in
\R$:
\begin{equation}
\P\left(\{Z \in [x , x + \varepsilon]\} \right) =
\E\left[ \P \left( \cA_{\boldsymbol{t}} \,  \big| \, \boldsymbol{t}
    \right) \right]
\end{equation}
where
\begin{equation}\label{def:ac}
\cA_{\boldsymbol{t}} := \left\{ \boldsymbol{\eta} \in \{0,1 \}^N \, : \,
               \Phi(\boldsymbol{Y}(\boldsymbol{t}) +
\boldsymbol{\delta}(\boldsymbol{t})\,  \boldsymbol{\eta}) \in [x,
x+\varepsilon] \right\} \, .
\end{equation}

By virtue of \eqref{condF}, the set $ \cA_{\boldsymbol{t}} $ is an
antichain in its dependence on $ \{ \eta_j \}_{j \in  J_{\boldsymbol{t}} } $
with
$  J_{\boldsymbol{t}} := \left\{ j \, : \, \delta_j(t_j) \geq x_+ - x_- \right\} $.
Lemma~\ref{lemprobsp}
thus yields
\begin{equation}
      \P \left\{ \cA_{\boldsymbol{t}} \,  \big| \, \boldsymbol{t}, \{
\eta_j \}_{j \not\in  J_{\boldsymbol{t}}} \right\}
      \leq \min\left\{ 1 , \frac{\Theta}{\sigma_\eta
\sqrt{|J_{\boldsymbol{t}}}|} \right\}
\end{equation}
   with $ \sigma_\eta = \sqrt{ p (1-p)} $. 
   We conclude  by the large-deviation argument used in the proof of
Lemma~\ref{lem:extsp}. Using \eqref{eq:delarge} the expected value of $ |J_{\boldsymbol{t}}| =
\sum_{j=1}^N 1_{\{  j \, : \, \delta_j(t_j) \geq x_+ - x_- \}} $
is bounded below:
\begin{equation}
     \E\left(|J_{\boldsymbol{t}}|\right) \geq (p_+ + p_-) N \, .
\end{equation}
Therefore $ \{ |J_{\boldsymbol{t}}| \leq \frac 1 2(p_+ + p_-) N  \} $ is a
large deviation event and its probability is exponentially bounded. 
Elementary estimates lead to 
\begin{equation}
\E\left(\min\left\{ 1 , \frac{\Theta}{\sigma_\eta
\sqrt{|J_{\boldsymbol{t}}}|} \right\} \right) \leq\frac{\widetilde
\Theta}{\sqrt{N}}\, \sqrt{\frac{1}{p_+} + \frac{1}{p_-}}\; ,
\end{equation}
with the same constant $ \widetilde \Theta $ as in \eqref{Sper_varied}.
\end{proof}

\begin{remark}
\begin{nummer}
\item  {\em A simpler proof for iid variables.\/}  For iid non-degenerate random variables $X_1, \ldots , X_N$ the theorem has a simpler proof 
using the binary decomposition of Theorem~\ref{thm:rep_t+}; there is no need for the large deviation argument.  The constants in the theorem will then depend on the value of $p$ and its corresponding lower bound in \eqref{infdelta}.
\item {\em The linear case.\/} For linear functions, 
\begin{equation}
 Z = \Phi(X_1,\dots, X_N) = \sum_{j=1}^N X_j \, , 
\end{equation}
concentration inequalities as in \eqref{concest} go back to 
 W. Doeblin, P. L\'evy~\cite{DL36,Doe39}, P. Erd\"os~\cite{Er} (for the Bernoulli case, where it reduces to the Littlewood-Offord problem),  A. N. Kolmogorov \cite{Kol58}, B. A. Rogozin \cite{Rog61}, H. Kesten \cite{Kes69} and C. G. Esseen \cite{Ess68}.
In this case, sharper inequalities than \eqref{concest} are known, e.g. \cite{Rog02}, 
\begin{equation}\label{eq:linconc}
  Q_Z(\varepsilon) \leq \Theta \, \varepsilon \, \Big[ \sum_{j=1}^N \varepsilon_j^2 \, (1- Q_{X_j}(\varepsilon_j)) \Big]^{-1/2}
\end{equation}
where $ \Theta $ is some constant.  
A recent application of the discrete case of the concentration bounds is found in \cite{TV}.   

\item {\em An extension.\/} 
As it  is already true for \eqref{eq:linconc}, the statement of
Theorem~\ref{thmconc} has an immediate extension to functions which
in some variables are monotone increasing and in some are monotone
decreasing, satisfying the natural analog of \eqref{condF}.   For
this extension, one only needs to replace  $p_+ $ and $p_-$ in
\eqref{concest} by $\hat p = \min\{p_+, p_-\}$.

\item {\em Sperner bounds from concentration inequalities.\/}
 In the proof of Theorem~\ref{thmconc}  concentration bounds 
were deduced from the probabilistic Sperner estimate   \eqref{probSpernerpq}.  
For antichains in the multiset 
$S= \{0, 1, \dots, K\}^N$ the implication  can also be
established in the opposite direction.   
For that, one may use the  fact that in such a multiset  
for any antichain 
$\cA$  there is a function $\Phi: S\mapsto \R$ which
satisfies the `representation condition' (in the terminology of
\cite{E})
\beq
\Phi(\boldsymbol{u} +\boldsymbol{e}_j) - \Phi(\boldsymbol{u}) \ \ge \ 1
\eeq
and for which $\Phi(\boldsymbol{u}) =0$ if and only in
$\boldsymbol{u} \in \cA  $.   
\end{nummer}
\end{remark}


\section{An Application to Random Schr\"odinger   Operators}  \label{sec:loc}

As a demonstration of a possible uses of the elementary observations
which are made in this article, let us present the case of spectral
localization under random iid single site potential  for an arbitrary probability 
distribution.

The  (continuum)  Anderson Hamiltonian is  the random Schr\"odinger 
operator given by
\begin{align}\label{Hom}
H_{\bom} =  -\Delta + 
V_{\bom} \quad \text{on} 
\quad
\mathrm{L}^2(\mathbb{R}^d),\\
\intertext{with}
V_{\bom} (x)= 
\sum_{\xi \in \Z^d} \omega_\xi \,  u(x - 
\xi),\label{Vom}
\end{align}
where  
\begin{enumerate} 
\item 
$u(\cdot)$, the single site potential, is a  nonnegative bounded 
measurable function
on $\R^{d}$ with compact support, uniformly 
bounded away from zero in
a neighborhood of the origin,
\item 

$\bom=\{ \omega_\xi \}_{\xi \in
\Z^d}$ is a family of independent 
identically distributed random
variables,  whose  common probability 
distribution $\mu$ satisfies:  $\set{0
,M}\in \supp \mu \subset 
[0,M]$,  for some $M>0$.
\end{enumerate} 

The random operator  $H_{\bom}$  is  a function of $\bom$, and as 
such it is defined  over a probability space which is invariant under 
the ergodic action of the group of $\Z^d$ translations.   The 
induced maps on this operator valued function are implemented by unitary translations.

Ergodicity considerations carry the  implication that there exist fixed 
subsets of $\R$ so 
that the spectrum of the self-adjoint operator $H_{\bom}$, as well as 
its pure point (pp),
absolutely continuous (ac), and singular continuous (sc) components, 
are equal to these fixed sets with probability one 
(c.f.~\cite{P,KS,KM}).
  In the case of the random potential \eqref{Vom}, the positivity of 
$u(\cdot)$ and  the support properties of $\mu$ imply that 
\beq  \label{eq:spectrum}
\sigma(H_{\bom}) \stackrel{as}{=} [0,\infty) \, .  
\eeq 

Although definitions of localization may come in several flavors, they all include (or imply)  spectral localization (i.e., pure point spectrum), as given in the following definition.

\begin{definition}\label{def:sploc}
 A self-adjoint  operator $H$ on $\mathrm{L}^2(\mathbb{R}^d)$ is said to exhibit spectral localization in   a closed interval 
$I \subset \R$
if  $\sigma(H)\cap I \neq\emptyset$ and the corresponding spectral 
projection $P_I(H)$ is given by a countable sum of orthogonal projections on proper eigenspaces. 
\end{definition}

This property is clearly invariant under translations.   The defining condition is equivalent to the requirement that for a spanning set of vectors the spectral measure is pure-point within $I$.   The set of $\bom$ for which this holds  for the random operator $H_{\bom}$ is known  to be measurable.

  In the one-dimensional case the continuous Anderson Hamiltonian has
been long known to exhibit spectral localization in the whole real line
 for any non-degenerate $\mu$, i.e. when the random 
potential is not constant \cite{GMP,DSS}.   In the multidimensional case, 
localization at
the bottom of the spectrum is already known at great, but nevertheless 
not all-inclusive, generality; cf.~\cite{St,Kle,BK} and references therein.
The Bernoulli decomposition presented here allows to prove localization for 
general non-degenerate single site distributions $\mu$.

More explicitly, the simplest  case to deal with, 
  for the different approaches which yield 
proofs of localization,  has  been   when the single site probability distribution  is absolutely continuous with bounded 
derivative.   The absolute continuity  condition can be relaxed 
 to H\"older continuity of $\mu$, 
both in the approach based on the multiscale analysis which was 
 introduced 
 in \cite{FS} and is discussed in ~\cite{Kle}, 
 and in the one based on the fractional moment method 
 of \cite{AM, AENSS}.  (The basis in the former case is an improved 
 analysis of 
 the Wegner estimate, which can be found in~\cite{St,CHK}.)    
 However, techniques relying on the regularity of $\mu$ 
 seem to reach their 
limit with log-H\"older continuity.  In particular, 
until recently  the Bernoulli random potential had been beyond 
the reach of analysis in more than one dimension.
For that extreme 
case, i.e., of $H_{\bom}$  with  $\mu\set{1}=
\mu\set{0}=\frac 1 2$, localization at the bottom 
of the spectrum  was recently proven by
Bourgain and Kenig \cite{BK}.
A crucial step in the analysis  of \cite{BK} is the
estimation of the probabilities of energy resonances  using Sperner's 
Lemma, i.e., the $ p = \tfrac{1}{2} $
version of \eq{probSpernerpq}.

The point which we would like to make here is that the Bernoulli 
decomposition of random variables  enables one to turn 
the latter result of Bourgain and Kenig \cite{BK} into a tool
for a general proof of localization at the edge of the spectrum for
arbitrary non-degenerate $\mu$.   

First, the Bourgain and Kenig \cite{BK} analysis needs to be extended 
to Schr\"odinger operators  which incorporate an additional 
background potential $ U \in L^\infty(\mathbb{R}^d) $,  and for which 
the variances of the Bernoulli terms are uniformly positive, thought 
not necessarily uniform.  More explicitly, the class  is broadened to 
include operators of the form  
\begin{equation}\label{eq:BKH} 
H_{\boldsymbol{\eta}} = - \Delta + U(x) + \sum_{\xi \in 
\mathbb{Z}^d} \eta_\xi \, b_\xi \, u(x-\xi)  \, ,
\end{equation}
where $u(\cdot)$ is as in \eqref{Vom}, satisfying 
the above condition (1), but instead of 
(2):
\begin{enumerate}
\item[(2')] $ \boldsymbol{\eta} = \{ \eta_\xi 
\}_{\xi \in \mathbb{Z}^d} $ are iid Bernoulli random variables taking 
the values $ \{0,1\} $ 
with probabilities $ \{1-p,p\} $, and the 
coefficients $ \{ b_\xi \}_{\xi \in 
\mathbb{Z}^d} $ 
satisfy
\begin{equation}
0 < b_- \leq  b_{\xi} \leq 
b_+ < \infty 
\qquad \mbox{for all $ \xi \in \Z^d$,}
\end{equation} 
\end{enumerate}
and 
\begin{enumerate}
\item[(3)] $U\in L^\infty (\R^d)$ satisfies, for all $ x \in \R^d$: 
$ 0 \le U(x) \le U_+ < \infty$. 
\end{enumerate}
Due to the presence of the background potential $U$ the spectrum of $H_{\boldsymbol{\eta}}$ need not be deterministic, i.e., equal to some  fixed set with probability one.  For our main purpose it would suffice to restrict attention to  $U$ for which the spectrum of $H_{\boldsymbol{\eta}}$  
is almost surely $[0,\infty)$.  Such restriction is not included in the following statement but instead there is a caveat in the conclusion.   

The extended BK  result, whose proof is  presented  in
\cite{GK3},  is:

\begin{theorem}\label{thmBK} 
Given  a function $u(\cdot)$
as above, and: $p\in (0,1)$,  $b_\pm > 0$ and $ U_+ < \infty$, 
there exist $E_0 >0$
such that any
random    operator $H_{\boldsymbol{\eta}} $ of the form  \eqref{eq:BKH},  
satisfying conditions (1), (2') and (3), for otherwise arbitrary external potential $U$,  with probability one, either exhibits  spectral localization in
$[0,E_0]$ or  $\sigma(H_{\boldsymbol{\eta}} ) \cap
[0,E_0] {=}\emptyset$.
\end{theorem}

Theorem~\ref{thm:rep_t+} allows now to deduce the following general
statement from the above non-trivial Bernoulli result.

\begin{theorem}\label{thm:gen_loc}  
Let   $H_{\boldsymbol{\omega}} = 
- \Delta + V_{\boldsymbol{\omega}} $   be a  Schr\"odinger operator 
with the  random potential given by \eqref{Vom}, satisfying 
the 
above conditions (1) and (2).  
Then for some $E_0>0$  the operator 
$H_{\boldsymbol{\omega}}$, with probability one, 
exhibits spectral localization in  $[0,E_0]$.
\end{theorem}

\begin{proof} 
The  Bernoulli decomposition \eqref{eq:t+} allows to write the 
coefficients in the random potential in the form: 
\begin{equation}\label{Berdecomega}
    \boldsymbol{\omega}\stackrel{\mathcal{D}}{=} \left\{ Y_+(t_\xi) +
\delta_{p}^{+}(t_\xi) \, \eta_\xi \right\}_{\xi \in \mathbb{Z}^d } \, , 
\end{equation}
with 
 $ \boldsymbol{t} =  \{ t_\xi \}_{\xi \in \mathbb{Z}^d} $ a family
of independent random variables which are uniformly  distributed in $ 
(0,1) $, 
 $ Y_+ $ and $ \delta_{p}^{+} $ the  functions defined in
\eqref{eq:delta} in terms of the distribution function of $ \mu $,  
and  $ \boldsymbol{\eta} = \{ \eta_\xi \}_{\xi \in \mathbb{Z}^d } $ 
a  family  of iid Bernoulli variables, independent of  $ \boldsymbol{t}$, which
take values
in $ \{ 0, 1 \} $ with probabilities $ \{ 1-p , p \} $ for some $ p 
\in (0,1) $ such that \eqref{infdelta} holds.
  
As a consequence, the random operator can be written  as: 
\begin{equation}\label{Hteta}
   H_{\boldsymbol{\omega}} \stackrel{\mathcal{D}}{=}   - \Delta +
U_{\boldsymbol{t}} + V_{\boldsymbol{t},\boldsymbol{\eta}} =:
H_{\boldsymbol{t},\boldsymbol{\eta}} 
   \end{equation}
where \begin{equation} 
 U_{\boldsymbol{t}}(x)  := \sum_{\xi \in \mathbb{Z}^d} Y_+(t_\xi)\,
u(x-\xi)  \quad \mbox{and} \quad 
    V_{\boldsymbol{t},\boldsymbol{\eta}}(x)  :=  \sum_{\xi \in
\mathbb{Z}^d} \delta_{p}^{+}(t_\xi) \, \eta_\xi \,  u(x-\xi) \, , 
\end{equation}
and the following bounds hold
\begin{align}
& 0 \le  U_{\boldsymbol{t}}(x) \le U_+:=  
M\,  \Bigg\lVert \sum_{\xi \in\mathbb{Z}^d}   u(\cdot-\xi) \Bigg\rVert_\infty < \infty
 \, ,\notag \\
& 0 < b_-:= \inf_{t \in (0,1)}  \delta_{p}^{+}(t) \le b_+:= M < \infty\, .
\end{align} 
  This 
implies that  when conditioned on the values of 
$ \boldsymbol{t} $ 
the operator $H_{\boldsymbol{t},\boldsymbol{\eta}} $
is of the form 
\eqref{eq:BKH}, with  $p$, $U_+$ and $b_\pm$ independent of  $ \boldsymbol{t} $.   
Thus, by   Theorem~\ref{thmBK} there exists $E_0 >0$ such that when conditioned 
on $ \boldsymbol{t} $  with probability one, 
$H_{\boldsymbol{t},\boldsymbol{\eta}} $
either exhibits  spectral localization or  has no spectrum in $[0,E_0]$.  
However, the latter is excluded (almost surely, also with respect to the conditional 
probability) by \eqref{eq:spectrum} and Fubini.  
 \end{proof}

\begin{remark} 
 In addition to the spectral localization it is also of interest to establish the existence of uniform localization length, i.e., to prove that 
all eigenfunctions $\phi$ of $H_{\bom}$ with eigenvalue in $[0,E_0] $
satisfy 
\beq  \int_{|x-y|\leq  \frac 1 2} | \phi(y) |^2 \, dy
 \le C_{\phi} \, e^{-2|x|/ \ell} \quad \text{for
all $x \in \R^{d}$}\, .
\eeq  
This can be accomplished in the following two ways, for which 
the details are  presented  in \cite{GK3}.   

To establish uniform localization length under the hypotheses of    Theorem~\ref{thm:gen_loc} 
one may use  the 
 Bernoulli decomposition  \eqref{Berdecomega} \emph{before} performing
 the multiscale analysis which is behind the proof of  Theorem~\ref{thmBK}. 
  The multiscale analysis is then executed for the random Schr\"odinger operator $H_{\boldsymbol{t},\boldsymbol{\eta} }$ in \eqref{Hteta},
 in such  a way that all events in the analysis are jointly measurable in $\boldsymbol{t}$ and $\boldsymbol{\eta}$.

An alternative proof of Theorem~\ref{thm:gen_loc}, which yields also uniform localization length,   can be based on the concentration bound of Theorem~\ref{thmconc}.   Namely, the  Bourgain-Kenig proof can be extended to arbitrary single site probability distribution $\mu$, with the probabilities of energy resonance estimated by the concentration bound instead of by Sperner's Lemma as in \cite{BK}  (see \cite{GK3}). 
   \end{remark}




\section*{Acknowledgements}   
We  thank the Oberwolfach center for  hospitality
at a meeting where the four-way collaboration 
started, and the Isaac Newton Institute where some of the work was done.  
We also thank B. Sudakov for an instructive review of   
the recent results in Sperner's theory, S. Molchanov for alerting us to 
relevant references, and M. Cranston for many helpful discussions 
of  results in probability theory.  
This work was supported in parts by the NSF grants DMS-0602360 (MA), DMS-0457474 (AK) and DMS-0701181 (SW), and a Rothschild Fellowship at INI (MA).

\end{document}